\documentclass[12pt]{article}

\usepackage{amsmath,amsthm,amsfonts,amssymb, color,xcolor,subcaption,graphicx} 

\usepackage{todonotes, dsfont}

\newtheorem{theorem}{Theorem}[section]

\newtheorem{lemma}[theorem]{Lemma}
\newtheorem{definition}[theorem]{Definition}

\def\cB{\mathcal{B}}
\def\cC{\mathcal{C}}

\def\cE{\mathcal{E}}
\def\cF{\mathcal{F}}
\def\cG{\mathcal{G}}
\def\cH{\mathcal{H}}

\def\cL{\mathcal{L}}

\def\cN{\mathcal{N}}
\def\cP{\mathcal{P}}

\def\cZ{\mathcal{Z}}

\def\bE{\mathbb{E}}

\def\bP{\mathbb{P}}
\def\bR{\mathbb{R}}

\def\k{\kappa}

\def\s{\sigma}

\newcommand{\fm}{\mathfrak{m}}

\begin{document}

\title{It\^o integral for a two-sided L\'evy process}

\author{Raluca M. Balan\footnote{Corresponding author. University of Ottawa, Department of Mathematics and Statistics, 150 Louis Pasteur Private, Ottawa, Ontario, K1N 6N5, Canada. E-mail address: rbalan@uottawa.ca.} \footnote{Research supported by a grant from the Natural Sciences and Engineering Research Council of Canada.}
\and
Jaime Garza \footnote{University of Ottawa, Department of Mathematics and Statistics, 150 Louis Pasteur Private, Ottawa, Ontario, K1N 6N5, Canada. E-mail address: jgarza@uottawa.ca.}
}

\date{May 12, 2026}
\maketitle

\begin{abstract}
\noindent In this article, we construct an It\^o integral with respect to a two-sided finite-variance L\'evy process $\{L(x)\}_{x\in \bR}$, without a Gaussian component.  Using Rosenthal inequality for discrete-time martingales, we give an estimate for the $p$-th moment of this integral, for any even integer $p\geq 2$. Then, using Poisson-Malliavin calculus, we show that the It\^o integral is an extension of the Hitsuda-Skorohod integral with respect to the compensated Poisson random measure associated to the L\'evy process.
\end{abstract}

\noindent {\em MSC 2020:} Primary 60G51 ; Secondary 60G60, 60H07

\vspace{1mm}

\noindent {\em Keywords:} L\'evy process, L\'evy white noise, Poisson random measure, It\^o integral, Poisson-Malliavin calculus



\section{Introduction}

Let $N$ be Poisson random measure on $\bR \times \bR_0$ of intensity $\fm(dx,dz)=dx \nu(dz)$,
defined on a complete probability space $(\Omega,\cF,\bP)$, where $\nu$ is a measure on $\bR_0$ which satisfies:
\[
m_2:=\int_{\bR_0}z^2 \nu(dz)<\infty.
\]
The space $\bR_0:=\bR \verb2\2 \{0\}$ is endowed with the distance $d(x,y)=|x^{-1}-y^{-1}|$. Let $\widehat{N}(F)=N(F)-\fm(F)$ be the compensated version of $N$. 

The goal of this short note is to construct an It\^o integral for a two-sided L\'evy process $\{L(x) \}_{x \in \bR}$ with the L\'evy-It\^o decomposition:
\[
L(x)={\rm sgn}(x) \left(b|x|+\int_{A_x\times \{|z| \leq 1\}} z \widehat{N}(dy,dz)+\int_{A_x \times \{|z|>1\}}z N(dy,dz)\right), \quad x\in \bR,
\]

and to discuss two of its properties: (i) a moment estimate; (ii) its relation with the Hitsuda-Skorohod integral. 
Here $A_x=[0,x]$ if $x \geq 0$ and $A_x=[x,0]$ if $x<0$.  
We assume that 
\[
b=-\int_{\{|z|>1\}} z\nu(dz),
\]
so that $L(x)$ is centered and has the representation:
\[
L(x)={\rm sgn}(x)\int_{A_x \times \bR_0}z \widehat{N}(dy,dz).
\]

While there are many references dedicated to stochastic analysis for L\'evy processes (e.g. \cite{applebaum09,BGJ87,DOP09,last16,NN,PZ07}), this specific construction and properties cannot be found in the literature. One can define a stochastic integral with respect to $\widehat{N}$ for any function  $f \in L^2(\bR \times \bR_0;\fm)$, and this integral (called $I_1(f)$ in Section \ref{section-Malliavin} below) is an isometry. Then, for any deterministic $\varphi \in L^2(\bR)$, we define 
\[
\int_{\bR} \varphi(x) L(dx)=\int_{\bR \times \bR_0} \varphi(x) z \widehat{N}(dx,dz).
\]
 The difficulty is that this integral is {\em not} a martingale, since $N$ has no time component.

\medskip

To address this problem, we consider a more general process:
\begin{equation}
\label{def-L}
L(A)=\int_{A \times \bR_0}z \widehat{N}(dx,dz) \quad \mbox{for all $A \in \cB_b(\bR)$},
\end{equation}
where $\cB_b(\bR)$ is the class of bounded Borel sets of $\bR$. 

\medskip

We say that $L=\{L(A);A \in \cB_b(\bR)\}$ is a {\em (time-independent) L\'evy white noise}. This process has the following properties: \\
(i) $L(\emptyset)=0$ and $L(\{x\})=0$ a.s. for any $x \in \bR$;\\
(ii) $\bE[L(A)]=0$ and $\bE|L(A)|^2=m_2|A|$, 
where $|A|$ is the Lebesgue measure of $A\in \cB_b(\bR)$; \\
(iii) $L(A_1),\ldots, L(A_k)$ are independent, for any disjoint sets $A_1,\ldots,A_k \in \cB_b(\bR)$;\\
(iv) for any $A \in \cB_b(\bR)$, $L(A)$ has characteristic function:
\begin{equation}
\label{ch-funct-L}
\bE(e^{i\theta L(A)})=\exp\left\{ |A| \int_{\bR_0} \big(e^{i\theta z}-1-i\theta z\big) \nu(dz) \right\} \quad \mbox{for any} \ \theta \in  \bR.
\end{equation}

Letting $L(x):=L([0,x])$ if $x \geq 0$ and $L(x):=-L([x,0])$ if $x<0$, we see that:
\[
L((x,y])=L(y)-L(x) \quad \mbox{for any} \quad x,y \in \bR, x<y.
\] 
Then $\{L(x)\}_{x\in \bR}$ is a L\'evy process, since it has stationary and independent increments, and $L(0)=0$ a.s. Moreover, $\bE[L(x)]=0$ and $\bE|L(x)|^2=m_2 |x|$ for any $x \in \bR$.

\medskip

This article is organized as follows. In Section \ref{section-Ito}, we give the construction of the It\^o integral with respect to $L$. In Section \ref{section-moment}, we prove the following moment estimate, which may be well-known but we could not locate it in the literature:
\[
\bE|L(x)|^p \leq C_p^* \left\{(m_2|x|)^{p/2}+m_p|x| \right\} \quad 
\mbox{where} \quad
m_p=\int_{\bR_0}|z|^p\nu(dz).
\]
Then, we extend this inequality to random integrands. 
In Section \ref{section-Malliavin}, we prove that the Hitsuda-Skorohod integral is an extension of the It\^o integral.

\section{It\^o integral}
\label{section-Ito}

In this section, we give the construction of the integral with respect to the L\'evy white noise $L$. This is similar to an It\^o integral, except that it is defined on $\bR$, instead of $[0,\infty)$.

\medskip

We let $(\cF_x)_{x\in \bR}$ be the filtration induced by $N$, given by:
\[
\cF_x=\s\big\{ N\big([a,b] \times \Gamma\big); -\infty<a<b\leq x,  \, \Gamma \in \cB_b(\bR_0) \big\} \vee \cN,
\]
where $\cN$ is the class of $\bP$-negligible sets.

A process $\Phi=\{\Phi(x)\}_{x\in \bR}$ is called {\em simple} if it is of the form
\begin{equation}
\label{simple}
\Phi(x)=\sum_{i=1}^n Y_i 1_{(b_{i-1},b_i]}(x)
\end{equation}
where $-\infty<b_0<b_1<\ldots<b_n<\infty$ and $Y_i$ is a bounded $\cF_{b_{i-1}}$-measurable for all $i=1,\ldots,n$. 
We denote by $\cE$ the class of simple processes. 

\begin{definition}
\label{def-pred}
{\rm
Let $\cP=\s(\cE)$ be the predictable $\sigma$-field on $\Omega \times \bR$. A process $\{\Phi(x)\}_{x\in \bR}$ is called {\em (spatially) predictable} if the map $(\omega,x) \mapsto \Phi(\omega,x)$ is $\cP$-measurable.
}
\end{definition}


\medskip

For a simple process $X$ of form \eqref{simple} and $K>0$, we define the stochastic integral with respect to $L$ on $[-K,K]$ by
\[
I_K(X):=\int_{-K}^K X(x)L(dx)=\sum_{i=1}^n Y_i L\big((b_{i-1},b_i] \cap [-K,K] \big).
\]
This integral has the following properties:
\begin{description}
\item[(i)] $\bE[I_K(X)]=0$ and $\bE|I_K(X)|^2=m_2 \big[X \big]_{K,2}^2$ for any $X \in \cE$, where
\[
\big[X \big]_{K,2}^2:=\bE \int_{-K}^K |X(x)|^2dx.
\]
\item[(ii)] $I_K(aX+bY)=aI_K(X)+bI_K(Y)$ for any $X,Y \in \cE$ and $a,b \in \bR$;

\item[(iii)] $I_{K}(X)=I_{K'}\big(X 1_{[-K,K]}\big)$ for any $X \in \cE$ and  $K'\geq K>0$.
\end{description}

Let $\cL_2$ be the set of predictable processes $X$ such that $[X]_{K,2}<\infty$ for all $K>0$.

\begin{lemma}
\label{approx}
If $X \in \cL_2$, then there exists a sequence $(X_m)_{m\geq 1}$ of simple processes such that 
 
$\big[X_m-X\big]_{K,2} \to 0$ as $m \to \infty$, for any $K>0$.
\end{lemma}

\begin{proof}
If $X$ is predictable and bounded, we use a similar argument as in the proof of Lemma 2.4 in Chapter 3 of \cite{KS}, 
replacing $[0,T]$ by $[-K,K]$. In the general case, we proceed as in the proof of Proposition 2.6 in Chapter 3, {\em ibid.} We omit the details.
\end{proof}

If $X \in \cL_2$ and $(X_m)_{m\geq 1}$ is the sequence from Lemma \ref{approx}, then $\{I_K(X_m)\}_{m\geq 1}$ is Cauchy in $L^2(\Omega)$. We denote its limit by $I_K(X)$. The limit does not depend on the sequence $(X_m)_{m\geq 1}$. Properties {\bf (i)}-{\bf (iii)} above continue to hold for processes in $\cL_2$. 

\medskip

To extend the integral to the entire space, we let $\cL_2^*$ be the set of predictable processes $X$ such that
\[
\big[X \big]_{2}^2:=\bE\int_{\bR}|X(x)|^2dx<\infty.
\]

If $X \in \cL_2^*$ then $\{I_K(X)\}_{K>0}$ is Cauchy in $L^2(\Omega)$, since for any $K'>K$,
\begin{align*}
\bE|I_{K'}(X)-I_K(X)|^2=\bE\big|I_{K'}(X)-I_{K'}\big(X1_{[-K,K]}\big)\big|^2=m_2\bE\int_{\{K<|x|\leq K'\}}|X(x)|^2 dx \to 0,
\end{align*}
as $K,K' \to \infty$. We denote $I(X):=\lim_{K \to \infty}I_K(X)$ in $L^2(\Omega)$.  We will use the notation:
\[
I(X)=\int_{\bR}X(x)L(dx).
\] 

Then $\bE[I(X)]=0$ and
$\bE|I(X)|^2 =m_2 \big[X \big]_{2}^2$.

\section{Moment estimate}
\label{section-moment}

In this section, we prove a moment estimate for the stochastic integral with respect to the L\'evy white noise $L$. Since the noise does not have a time component, we cannot use Rosenthal's inequality for continuous martingales, as in \cite{BZ24,BZ26}. Instead, we develop a moment inequality for $L(\varphi)$ for any $\varphi \in L^2(\bR)$, then for the stochastic integral of a simple process (using Rosenthal's inequality for discrete martingales), and finally for the stochastic integral of a general process (by approximation with simple processes).

\medskip

For any $p\geq 1$, we 
we denote by $\|\cdot\|_p$ the norm in $L^p(\Omega)$, and
\[
m_p:=\int_{\bR_0}|z|^p \nu(dz).
\]

We recall that for any integer $m\geq 2$, the $m$-th moment of a random variable $X$ is given by the {\em cumulant formula}:
\[
\bE[X^m]=\sum_{\pi \in \cP_m} \prod_{B \in \pi} \k_{|B|},
\]
where $\cP_{m}$ is the set of all partitions of $\{1,\ldots,m\}$, ``$B \in \pi$'' means that $B$ is one of the sets (or ``blocks'') of partition $\pi$, $|B|$ is the size of $B$, and $\k_n$ is the {\em $n$-th cumulant} of $X$:
\[
\k_n=(-i)^n H^{(n)}(0) \quad \mbox{where} \quad H(\theta)=\log \bE[e^{i\theta X}],\theta \in \bR.
\]
If we let $\mu=\bE(X)$ and $\mu_n'=\bE[(X-\mu)^n]$ for any integer $n \geq 1$, then
\[
\k_1=\mu, \quad \k_2=\mu_2', \quad \k_3=\mu_3', \quad \k_4=\mu_4'-3\mu_3'.
\]

Note that if $\bE[X]=0$, then $\k_1=0$ and hence,
\begin{equation}
\label{cum}
\bE[X^m]=\sum_{\pi \in \cP_m^*} \prod_{B \in \pi} \k_{|B|},
\end{equation}
where $\cP_m^*$ is the set of all partitions $\pi \in \cP_m$ which contain blocks $B$ with $|B|\geq 2$.

\medskip

\begin{lemma}
\label{lem-p-mom-L}
Let $p\geq 2$ be an even integer such that $m_p<\infty$. For any $\varphi \in L^2(\bR)$,
\begin{equation}
\label{p-mom-L}
\bE|L(\varphi)|^p \leq C_p^* \left\{ \left(m_2\int_{\bR}|\varphi(x)|^2 dx\right)^{p/2} + m_p \int_{\bR}|\varphi(x)|^p dx\right\},
\end{equation}
where  $C_p^*={\rm card}(\cP_p^*)$.
In particular, for any set $A \in \cB_b(\bR)$,
\begin{equation}
\label{p-mom-L}
\bE|L(A)|^p \leq C_p^* \left\{ \big(m_2|A|\big)^{p/2} + m_p |A| \right\}.
\end{equation}
\end{lemma}

\begin{proof}
By \eqref{ch-funct-L} and Taylor formula,
\[
H(\theta):=\log \bE[e^{i \theta L(A)}]=\int_{\bR_0}\int_{\bR}\big(e^{i \theta \varphi(x)z} -1-i\theta \varphi(x) z\big) dx\nu(dz)=\sum_{n \geq 2} \frac{\big(i\theta \big)^n}{n!} \widetilde{m}_n \int_{\bR} \varphi(x)^n dx,
\]
where $\widetilde{m}_n=\int_{\bR_0} z^n \nu(dz)$. Hence, $H^{(n)}(0)=i^n \widetilde{m}_n \int_{\bR} \varphi(x)^n dx$, and the $n$-th cumulant of $L(\varphi)$ is 
\[
\k_n= \widetilde{m}_n \int_{\bR} \varphi(x)^n dx \quad \mbox{for any $n\geq 2$}.
\]

Since $p$ is even, by the cumulant formula \eqref{cum},
\begin{equation}
\label{mom-LA1}
0\leq \bE|L(\varphi)|^p=\bE[L(\varphi)^p]=\sum_{\pi \in \cP_p^*} \prod_{B \in \pi} \k_{|B|},
\end{equation}

Let $\pi \in \cP_p^*$ be arbitrary. Assume that $\pi$ has blocks $B_1,\ldots,B_k$ with $|B_i|=b_i$ for all $i=1,\ldots,k$. Then $\sum_{i=1}^k b_i=p$ and $b_i \geq 2$. It follows that
\[
0\leq \prod_{B \in \pi} \k_{|B|}=\prod_{i=1}^{k}\k_{b_i}=\prod_{i=1}^{k}\widetilde{m}_{b_i} \int_{\bR} \varphi(x)^{b_i}dx \leq  \prod_{i=1}^{k}m_{b_i} \int_{\bR} |\varphi(x)|^{b_i}dx .
\]

Next, we provide an estimate for $m_{b_i}$, based on $m_2$ and $m_p$.
Any $r \in \{3,\ldots,p-1\}$ can be written as $r=p\theta+2(1-\theta)$ for some $\theta \in (0,1)$.  By H\"older's inequality,
\[
m_r=\int_{\bR_0} \big(|z|^{p} \big)^{\theta} \big(|z|^{2} \big)^{1-\theta}\nu(dz) \leq 
\Big(\int_{\bR_0} |z|^{p} \nu(dz) \Big)^{\theta} \Big(\int_{\bR_0}|z|^{2} \nu(dz)\Big)^{1-\theta}=m_{p}^{\theta}m_2^{1-\theta}.
\]
This inequality holds trivially for $r=2$ (when $\theta=0$) and  $r=p$ (when $\theta=1$). Hence
\[
m_r \leq m_{p}^{\theta} m_2^{1-\theta}  \quad \mbox{for any $r \in \{2,\ldots,p\}$, with $\theta=\frac{r-2}{p-2}$.} 
\]

Similarly, for any $r \in \{2,\ldots,p\}$, letting $\theta=\frac{r-2}{p-2}$, we have:
\[
\int_{\bR} |\varphi(x)|^r dx \leq \left( \int_{\bR} |\varphi(x)|^p dx \right)^{\theta}
\left( \int_{\bR} |\varphi(x)|^2 dx \right)^{1-\theta}. 
\]

Using this estimate, we obtain that
\begin{align}
\nonumber
\prod_{B \in \pi} \k_{|B|} & \leq  \prod_{i=1}^{k}\left(m_{p}  \int_{\bR} |\varphi(x)|^p dx\right)^{\frac{b_i-2}{p-2}} \left(m_{2} \int_{\bR} |\varphi(x)|^2 dx\right)^{\frac{p-b_i}{p-2}}\\
\nonumber
&= \Big[\Big(m_2 \int_{\bR} |\varphi(x)|^2 dx \Big)^{\frac{p}{2}}\Big]^{\frac{2k-2}{p-2}} \cdot \big(m_p \int_{\bR} |\varphi(x)|^p dx \big)^{\frac{p-2k}{p-2}} \\
\nonumber
&\leq \frac{2k-2}{p-2} \Big(m_2 \int_{\bR} |\varphi(x)|^2 dx \Big)^{p/2}+\frac{p-2k}{p-2} m_p \int_{\bR} |\varphi(x)|^p dx\\
\label{mom-LA2}
&\leq  \Big(m_2  \int_{\bR} |\varphi(x)|^2 dx\Big)^{p/2}+m_p  \int_{\bR} |\varphi(x)|^p dx,
\end{align}
using the inequality $x^a y^{1-a} \leq ax+(1-a)y$ for any $a \in [0,1]$ and $x,y>0$.

Relation \eqref{p-mom-L} follows from \eqref{mom-LA1} and \eqref{mom-LA2}.
\end{proof}

Next, we will extend this moment inequality to simple processes. For this, we use {\em Rosenthal's inequality} for discrete-time martingales, which we recall below. Let $(S_k)_{k=0,1,\ldots,n}$ be a martingale with respect to a filtration $(\cF_k)_{k=0,1,\ldots,n}$, such that $S_0=0$. Let $X_i=S_i-S_{i-1}$ for all $i=1,\ldots,n$. If $p \geq 2$ is such that $\bE|X_i|^p<\infty$ for all $i=1,\ldots,n$, then
\begin{equation}
\label{ros}
\bE\big[\max_{k\leq n}|S_k|^p\big] \leq \cB_p^p \left\{ \bE \Big[ \Big( \sum_{i=1}^n \bE[X_i^2|\cF_{i-1}] \Big)^{p/2} \Big]+\sum_{i=1}^{n}\bE|X_i|^p \right\},
\end{equation}
where $\cB_p>0$ is a constant depending on $p$; see \cite{burkholder73, merlevede-peligrad13}.

\begin{lemma}
\label{lem-simple}
Let $p \geq 2$ be an even integer such that $m_{p}<\infty$.
For any $X \in \cE$,
\[
\bE\left|\int_{\bR}X(x)L(dx) \right|^p \leq \cC_p^p \left\{
\bE\left[ \left( \int_{\bR}|X(x)|^2 dx\right)^{p/2}\right]+ \bE \int_{\bR}|X(x)|^p dx\right\},
\]
where $\cC_p^p=2 \cB_p C_p^* (m_2^{p/2} \vee m_p)$, and $\cB_p,C_p^*$ are the constants from \eqref{ros}, respectively \eqref{p-mom-L}.
\end{lemma}

\begin{proof}
Let $X$ be a simple process of form \eqref{simple}. Then 
\[
\int_{\bR}X(x)L(dx)=\sum_{i=1}^n Y_i L(A_i) \quad \mbox{where} \quad A_i=(b_{i-1},b_i].
\] 
Let $S_k=\sum_{i=1}^k Y_i L(A_i)$ for $k=1,\ldots,n$ and $S_0=0$. Then $(S_k)_{k=0,\ldots,n}$ is a zero-mean martingale with respect to $(\cF_{b_k})_{k=0,\ldots,n}$: 
for any $i=1,\ldots,n$,
\[
\bE[Y_i L(A_i)|\cF_{b_{i-1}}]=Y_i \bE[L(A_i)|\cF_{b_{i-1}}]=Y_i \bE[L(A_i)]=0,
\] 
since $L(A_i)$ is independent of $\cF_{b_{i-1}}$. Similarly, $\bE \big[Y_i^2 L(A_i)^2|\cF_{b_{i-1}}\big] =Y_i^2 \bE|L(A_i)|^2=Y_i^2 m_2 |A_i|$.
Then, using Rosenthal's inequality \eqref{ros}, followed by estimate \eqref{p-mom-L}, we have:
\begin{align*}
\bE|S_n|^p & \leq \cB_p^p \left\{ \bE \Big[ \Big( \sum_{i=1}^n \bE \big[Y_i^2 L(A_i)^2|\cF_{b_{i-1}}\big] \Big)^{p/2} \Big]+\sum_{i=1}^{n}\bE \big|Y_i L(A_i)\big|^p \right\}\\
&=\cB_p^p \left\{ m_2^{p/2} \bE \Big[ \Big( \sum_{i=1}^n |Y_i|^2 |A_i|\Big)^{p/2} \Big]+\sum_{i=1}^{n} \bE |Y_i |^p \, \bE| L(A_i)|^p \right\} \\
& \leq \cB_p^p \left\{ m_2^{p/2} \bE \Big[ \Big( \sum_{i=1}^n  |Y_i|^2 |A_i|\Big)^{p/2} \Big]+\cC_p^* \sum_{i=1}^{n} \bE |Y_i |^p \big[ (m_2 |A_i|)^{p/2}+ m_p |A_i| \big] \right\}.
\end{align*}
Using the fact that $\bE \big[\sum_{i=1}^n |Y_i|^p |A_i|^{p/2}\big] \leq \bE\big[\big( \sum_{i=1}^n |Y_i|^2 |A_i| \big)^{p/2}\big]$, we obtain that
\begin{align*}
\bE|S_n|^p & \leq \cC_p^p \left\{ \bE \Big[ \Big( \sum_{i=1}^n  |Y_i|^2 |A_i|\Big)^{p/2} \Big]+ \sum_{i=1}^{n}\bE|Y_i|^p |A_i|\right\}.
\end{align*}

On the other hand, since the sets $A_1,\ldots,A_n$ are disjoint, 
\[
\int_{\bR}|X(x)|^2 dx=\sum_{i=1}^{n} Y_i^2|A_i| \quad \mbox{and} \quad \int_{\bR}|X(x)|^p dx=\sum_{i=1}^{n} Y_i^p|A_i|.
\]
The conclusion follows.
\end{proof}

Finally, we extend the moment inequality to the stochastic integral of a general process.
From Lemma \ref{lem-simple}, we infer that for any even integer $p\geq 2$ such that $m_p<\infty$,
\[
\|I(X)\|_p \leq \cC_p \big[X\big]_p \quad \mbox{for any $X\in \cE$},
\]
where
\begin{align*}
\big[X\big]_p 
& :=\|X\|_{L^p(\Omega;L^2(\bR))}+\|X\|_{L^p(\Omega \times \bR)}.
\end{align*}

Let $\cL_p^*$ be the set of predictable processes $X$ such that $\big[X \big]_p<\infty$. Then $\cL_p^* \subseteq \cL_2^*$, since $\|X\|_{L^2(\Omega;H)} \leq \|X\|_{L^p(\Omega;H)}$ with $H=L^2(\bR)$.
Similarly, $\cL_p \subseteq \cL_2$, where $\cL_p$ is the set of predictable processes $X$ such that $\big[ X \big]_{K,p}<\infty$ for all $K>0$, and
\[
\big[X\big]_{K,p}:=\|X\|_{L^p(\Omega;L^2([-K,K]))}+\|X\|_{L^p(\Omega \times [-K,K])}.
\]


The next result provides an extension of Lemma \ref{lem-simple} to processes in $\cL_p$.

\begin{lemma}
\label{lem-simple-p}
Let $p \geq 2$ be an even integer such that $m_p<\infty$. If $X \in \cL_p$, then
\begin{equation}
\label{p-mom-X}
\|I_K(X)\|_p \leq \cC_p \big[ X\big]_{K,p} \quad \mbox{for any} \ K>0.
\end{equation}
\end{lemma}

\begin{proof}
Similarly to Lemma \ref{approx}, it can be proved that there exists a sequence $(X_m)_{m\geq 1}$ of simple processes such that $\big[X_m-X \big]_{K,p} \to 0$ as $m \to \infty$, for any $K>0$. 

Fix $K>0$. Since $\big[\cdot \big]_{K,2} \leq \big[ \cdot \big]_{K,p}$, it follows that
$\big[X_m-X \big]_{K,2} \to 0$ as $m \to \infty$. Hence $I_K(X_m) \to I_K(X)$ in $L^2(\Omega)$ as $m \to \infty$. On the other hand, $\{I_K(X_m)\}_{m\geq 1}$ is a Cauchy sequence in $L^p(\Omega)$, since by Lemma \ref{lem-simple}, 
\[
\|I_K(X_m)-I_K(X_n)\|_p \leq \cC_p \big[X_n-X_m \big]_{K,p} \to 0 \quad \mbox{as $n,m \to \infty$}.
\]

Therefore, $\{I_K(X_m)\}_{m\geq 1}$ converges in $L^p(\Omega)$. Its limit must be $I_K(X)$. By Lemma \ref{lem-simple},
\[
\|I_K(X_m)\|_p \leq \cC_p \big[ X_m \big]_{K,p} \quad \mbox{for any} \ m\geq 1.
\] 
The conclusion follows letting $m \to \infty$.
\end{proof}

Finally, we provide an extension of this result to the stochastic integral on $\bR$.

\begin{theorem}
\label{p-mom-X1}
Let $p \geq 2$ be an even integer such that $m_p<\infty$. For any $X \in \cL_p^*$,
\begin{equation}
\label{IXp}
\|I(X)\|_p \leq \cC_p \big[ X\big]_{p}.
\end{equation}
\end{theorem}

\begin{proof}
Note that $\{I_K(X)\}_{K>0}$ is Cauchy in $L^p(\Omega)$: 
\begin{align*}
& \|I_{K'}(X)-I_K(X)\|_p =\|I_{K'}(X)-I_{K'}(X1_{[-K,K]})\|_p \leq \cC_p \big[X-X 1_{[-K,K]}\big]_{K',p}
\\
& \quad = \cC_p \left\{ \left\{\bE\left[ \Big(\int_{\{K<|x| \leq K'\}}|X(x)|^2 dx \Big)^{p/2} \right] \right\}^{1/p} + \left\{ \bE \int_{\{K<|x|\leq K'\}}|X(x)|^p dx \right\}^{1/p}\right\} \to 0,
\end{align*}
as $K,K' \to \infty$ and $K'>K$. Therefore $\{I_K(X)\}_{K>0}$ converges in $L^p(\Omega)$ as $K \to \infty$. Its limit must be $I(X)$ (since $X \in \cL_2^*$). The conclusion follows letting $K \to \infty$ in \eqref{p-mom-X}.
\end{proof}

As an application of Theorem \ref{p-mom-X1}, we obtain the following result for a mixed Lebesgue-stochastic convolution, which contains $ds$ and $L(dy)$. 

\begin{theorem}
\label{p-mom-X2}
Let $p\geq 2$ be an even integer such that $m_p<\infty$. Let $\Phi=\{\Phi(t,x);t\geq 0,x\in \bR\}$ be a random field such that $(\omega,x,t) \mapsto \Phi(\omega,t,x)$ is $\cP \otimes \cB(\bR_{+})$-measurable, and
\[
\sup_{t\leq T} \sup_{x\in \bR} \bE|\Phi(t,x)|^p<\infty \quad \mbox{for any $T>0$}.
\]
Let $(t,x)\mapsto \cG_t(x)$ be a measurable function on $\bR_{+} \times \bR$ such that 
\begin{equation}
\label{def-nuT}
\nu_T:=\int_0^T \int_{\bR} |\cG_{t}^p(x)|dxdt<\infty \quad \mbox{for any $T>0$}.
\end{equation}
Then, for any $t>0$ and $x \in \bR$,
\[
\bE \left| \int_0^t \int_{\bR} \cG_{t-s}(x-y) \Phi(s,y) ds L(dy) \right|^p \leq B_{t,p}^p
\int_0^t \int_{\bR} \big(\cG_{t-s}^2(x-y)+\cG_{t-s}^p(x-y)\big) \bE|\Phi(s,y)|^p dyds,
\]
where 
\begin{equation}
\label{def-Btp}
B_{t,p}^p=2^{p-1}\cC_p^p \big(t^{p/2}\nu_t^{p/2-1} + t^{p-1} \big),
\end{equation} 
and $\cC_p$ is the constant from \eqref{IXp}.
\end{theorem}

\begin{proof}
Let $\Psi(y)=\int_0^t \cG_{t-s}(x-y)\Phi(s,y)ds$
By Fubini theorem, $\Psi$ is predictable. By H\"older's inequality,
\[
\bE \int_{\bR}|\Psi(y)|^p dy \leq t^{p-1} \int_0^t  \int_{\bR}|\cG_{t-s}^p(x-y)|\bE|\Phi(s,y)|^pdyds<\infty.
\]
Hence, $\Psi \in \cL_p^*$.
By Theorem \ref{p-mom-X1},
\[
\bE|I(\Psi)|^p \leq 2^{p-1}\cC_p^p
\Big(\|\Psi\|_{L^p(\Omega;L^2(\bR))}^p+\|\Psi\|_{L^p(\Omega \times \bR)}^p \Big)=:2^{p-1}\cC_p^p(T_1+T_2).
\]

We estimate separately the two terms. First note that,
\[
\| \Psi(y)\|_p \leq \int_0^t |\cG_{t-s}(x-y)| \|\Phi(s,y)\|_p ds
\]
and hence, by H\"older's inequality,
\[
\| \Psi(y)\|_p^2  \leq t \int_0^t \cG_{t-s}^2 (x-y) \|\Phi(s,y)\|_p^2 ds, 
\quad 
\| \Psi(y)\|_p^p  \leq t^{p-1} \int_0^t |\cG_{t-s}^p (x-y)| \|\Phi(s,y)\|_p^p ds.
\]

By Minkowski's inequality, and H\"older's inequality with respect to the measure $\cG_{t-s}^2(x-y)dsdy$ on $[0,t] \times \bR$ (whose total measure is $\nu_t$),
we obtain that:
\begin{align*}
T_1 & =\left\| \int_{\bR} \Psi^2(y)dy\right\|_{p/2}^{p/2} \leq
\left( \int_{\bR} \|\Psi(y)^2\|_{p/2} dy \right)^{p/2} = \left( \int_{\bR} \|\Psi(y)\|_p^2 dy \right)^{p/2}\\
& \leq t^{p/2} \left( \int_0^t \int_{\bR} \cG_{t-s}^2 (x-y) \|\Phi(s,y)\|_p^2 dyds \right)^{p/2} \\
& \leq t^{p/2}\nu_{t}^{p/2-1} \int_0^t \int_{\bR} \cG_{t-s}^2 (x-y) \|\Phi(s,y)\|_p^{p} dyds.
\end{align*}

For $T_2$, we have:
\begin{align*}
T_2 & =\int_{\bR}\|\Psi(y)\|_p^p dy \leq t^{p-1} \int_0^t \int_{\bR} \cG_{t-s}^p(x-y) \|\Phi(s,y)\|_p^p dyds.
\end{align*}
\end{proof}

\section{Relation with  Hitsuda-Skorohod integral}
\label{section-Malliavin}

In this section, we show that the It\^o integral coincides with the Hitsuda-Skorohod integral for predictable integrands.

We begin by recalling some basic material about Malliavin calculus with respect to the compensated Poisson random measure $\widehat{N}$. Let $\cH=L^2(Z,\cZ,\fm)$, where
\[
({\bf Z},\cZ)=\big(\bR \times \bR_0,\ \cB(\bR) \otimes \cB(\bR_0)\big).
\]

$\bullet$ {\bf Chaos Expansion.}
Any random variable $F \in L^2(\Omega)$ which is $\cF^N$-measurable has the {\em Poisson-chaos expansion}:
\begin{equation}
\label{Poisson-chaos}
F=\bE(F)+\sum_{n\geq 1}I_n(f_n), \quad \mbox{for some $f_n \in \cH^{\odot n}$}.
\end{equation}
This series is orthogonal in $L^2(\Omega)$. Here $I_n$ is the multiple integral with respect to $\widehat{N}$ and $\cH^{\odot n}$ is the set of symmetric functions in $\cH^{\otimes n}$.
For any $f \in \cH^{\otimes n}$,
\[
\bE[I_n(f)]=0 \quad \mbox{and} \quad \bE|I_n(f)|^2=n! \|\widetilde{f}\|_{\cH^{\otimes n}}^2,
\]
where $\widetilde{f}$ is the symmetrization of $f$. Moreover, $I_n(f)=I_n(\widetilde{f})$ for any $f \in \cH^{\otimes n}$.

\medskip 
$\bullet$ {\bf Malliavin Derivative.}
For any random variable $F \in L^2(\Omega)$ with chaos expansion \eqref{Poisson-chaos}, we define the {\em Malliavin derivative} of $F$ by:
\[
D_{\xi}F=\sum_{n\geq 1}nI_{n-1}\big(f_n(\cdot,\xi)\big), \quad \mbox{for all} \quad \xi \in {\bf Z},
\]
provided that 
\[
\bE\|DF\|_{\cH}^2=\sum_{n\geq 1}nn! \|\widetilde{f}_n\|_{\cH^{\otimes n}}^2<\infty.
\] In this case, we write $F \in {\rm dom}(D)$.

\medskip

$\bullet$ {\bf Hitsuda-Skorohod integral.}
Let $\delta:{\rm dom}(\delta) \to L^2(\bR)$ be the adjoint of $D$, where ${\rm dom}(\delta)$ is the set of $V \in L^2(\Omega;\cH)$ for which there exists a constant $C=C_V>0$ depending on $V$, such that 
\[
\big|\bE \langle DF,V \rangle_{\cH}\big| \leq C \|F\|_2 \quad \mbox{for any $F \in {\rm dom}(D)$}.
\] 
We say that $\delta(V)$ is 
{\em the Hitsuda-Skorohod integral} of $V$ with respect to $\widehat{N}$, and we write
\[
\delta(V)=\int_{\bR} \int_{\bR_0}V(x,z)\widehat{N}(\delta x, \delta z).
\]
By duality, for any $V \in {\rm dom}(\delta)$,
\[
\bE\langle DF,V \rangle_{\cH}=\bE[F \delta(V)] \quad \mbox{for any $F \in {\rm dom}(D)$}.
\]

\begin{lemma}
\label{cond-Ik}
Let $h \in \cH^{\otimes k}$ be arbitrary. Then, for any $y \in \bR$,
\[
\bE[I_k(h)|\cF_y]=I_k(h^y), 
\]
where $h^y(x_1,z_1,\ldots,x_k,z_k)=h(x_1,z_1,\ldots,x_k,z_k) 1_{(-\infty,y]^k}(x_1,\ldots,x_k)$.
\end{lemma}

\begin{proof}
Let $\cE_k$ be the set of functions of the form
\begin{equation}
\label{def-h}
h(\xi_1,\ldots,\xi_k)=\sum_{i_1,\ldots,i_k=1}^{m}\beta_{i_1,\ldots,i_k} 1_{F_{i_1}\times\ldots \times F_{i_k}}(\xi_1,\ldots,\xi_k), \quad \mbox{with} \quad \xi_i=(x_i,z_i) \in {\bf Z},
\end{equation}
where $\beta_{i_1,\ldots,i_k}=0$ if $i_j=i_{\ell}$ for some $j \not=\ell$ and $F_1,\ldots,F_m\in \cZ$ are disjoint of the form $F_i=A_i \times B_i$ with $A_i \in \cB_b(\bR)$ and $B_i \in \cB_b(\bR_0)$. Thus, the sum in \eqref{def-h} is based only on {\em distinct} indices $i_1,\ldots, i_k$.
Since $\cE_k$ is dense in $\cH^{\otimes k}$,
it is enough to prove the result for $h \in \cE_k$. If $h$ is of form \eqref{def-h}, then
\[
I_k(h)=\sum_{i_1,\ldots,i_k=1}^{m}  \beta_{i_1,\ldots,i_k} \prod_{j=1}^{k}\widehat{N}(F_{i_j}),
\]
and
\[
\bE[I_k(h)|\cF_y]=\sum_{i_1,\ldots,i_k=1}^{m}\beta_{i_1,\ldots,i_k}\bE \Big[\prod_{j=1}^{k} \widehat{N}(F_{i_j}) \big|\cF_y \Big].
\]
For any $i=1,\ldots,m$, we write
\[
\widehat{N}(F_i)=\widehat{N}(A_i \times B_i)=\widehat{N}\big((A_i \cap (-\infty,y]) \times B_i\big)+
\widehat{N}\big((A_i \cap (y,\infty))  \times B_i\big)=:X_i+Y_i.
\]
Note that $X_1,\ldots,X_m$ are $\cF_y$-measurable, while $Y_1,\ldots,Y_m$ are independent of $\cF_y$. Moreover, $Y_1,\ldots,Y_m$ are independent, since the sets $F_1,\ldots,F_m$ are disjoint. 

Let $i_1,\ldots,i_k\in \{1,\ldots,m\}$ be distinct. Then
\[
\prod_{j=1}^{k} \widehat{N}(F_{i_j})=\prod_{j=1}^{k} \big(X_{i_j}+Y_{i_j}\big)=\sum_{J \subset \{1,\ldots,k\}} \prod_{j\in J} X_{i_j} \prod_{j \in J^c} Y_{i_j},
\]
and hence,
\[
\bE\Big[\prod_{j=1}^{k} \widehat{N}(F_{i_j})\big|\cF_y\Big]=\sum_{J \subset \{1,\ldots,k\}} \prod_{j\in J} X_{i_j} \bE\Big[\prod_{j \in J^c} Y_{i_j}\Big].
\]

If $J^c$ is non-empty, $\bE\Big[\prod_{j \in J^c} Y_{i_j}\Big]=\prod_{j \in J^c}\bE[Y_{i_j}]=0$ since $i_1,\ldots, i_k$ are distinct and $Y_1,\ldots,Y_m$ are independent. Therefore, the only term which contributes to the sum above if the one corresponding to $J=\{1,\ldots,k\}$. Hence,
\[
\bE[I_k(h)|\cF_y]=\sum_{i_1,\ldots,i_k=1}^{m}\beta_{i_1,\ldots,i_k}\prod_{j=1}^{k}
\widehat{N}\big((A_{i_j} \cap (-\infty,y]) \times B_{i_j}\big).
\]

On the other hand, $h^y(x_1,z_1,\ldots,x_k,z_k)=\sum_{i_1,\ldots,i_k=1}^{m}\beta_{i_1,\ldots,i_k} 
\prod_{j=1}^{k}1_{A_{i_j} \cap (-\infty,y]}(x_j) 1_{B_{i_j}}(z_j)$,
and hence,
\[
I_k(h^y)=\sum_{i_1,\ldots,i_k=1}^{m}\beta_{i_1,\ldots,i_k}  \widehat{N}\big( (A_{i_j} \cap (-\infty,y]) \times B_{i_j} \big).
\]
The conclusion follows.
\end{proof}

\begin{lemma}
\label{Fa-meas}
If $F \in {\rm dom}(D)$ is such that $F$ is $\cF_y$-measurable for some $y \in \bR$, then $D_{x,z}F=0$ for any $x>y$ and $z \in \bR_0$.
\end{lemma}

\begin{proof}
Assume that $F$ has the chaos expansion \eqref{Poisson-chaos}. By Lemma \ref{cond-Ik},
\[
F=\bE[F|\cF_y] =\bE(F)+ \sum_{n\geq 1} \bE[I_n(f_n)|\cF_y]=\bE(F)+\sum_{n\geq 1}I_n(f_n^y).
\]
Then, for any $x>y$ and $z \in \bR_0$, $D_{x,z}F=\sum_{n\geq 1}n I_{n-1}\big(f_n^y(\cdot,x,z)\big)=0$, since
\[
f_n^y(x_1,z_1,\ldots,x_{n-1},z_{n-1},x,z)=f_n(x_1,z_1,\ldots,x_{n-1},z_{n-1},x,z)\prod_{j=1}^{n-1}1_{(-\infty,y]}(x_j)
1_{(-\infty,y]}(x)=0.
\] 
\end{proof}

The following result gives the relation between the Hitsuda-Skorohod integral with respect to $\widehat{N}$ and the It\^o integral with respect to $L$, which originates from definition \eqref{def-L} of $L$.

\begin{lemma}
\label{Ito-Skor}
Let $\Phi \in \cL_2^*$ be arbitrary. Define $V_{\Phi}(x,z)=\Phi(x)z$ for all $x \in \bR$ and $z \in \bR_0$. Then $V_{\Phi} \in {\rm dom}(\delta)$ and $\delta(V_{\Phi})=I(\Phi)$.
\end{lemma}

\begin{proof}
We argue as in the proof of Theorem 10.2.7 of \cite{NN}.

\medskip

{\em Step 1.} Assume that $\Phi$ is a simple process of form \eqref{simple} where  $Y_{i} \in {\rm dom}(D)$ for all $i=1,\ldots,n$. Let $u_i(x,z)=z 1_{(b_{i-1},b_i]}(x)$. Fix $i \in \{1,\ldots,n\}$. By Proposition 10.2.6 of \cite{NN}, $Y_i u_i \in {\rm dom}(\delta)$ and
\[
\delta(Y_i u_i)=y_i \delta(u_i)-\langle DY_i,u_i \rangle_{\cH}-\delta(u_i DY_i).
\]
Since $Y_i$ is $\cF_{b_{i-1}}$-measurable, $D_{x,z}Y_i=0$ for all $x>b_{i-1}$ (by Lemma \ref{Fa-meas}). Hence,
\[
D_{x,z} Y_i z 1_{(b_{i-1},b_i]}(x)=0 \quad \mbox{for all $(x,z)\in \bR \times \bR_0$.}
\]
It follows that $\langle DY_i,u_i \rangle_{\cH}=0$ and $u_i DY_i=0$. Therefore,
\[
\delta(Y_i u_i)=Y_i \delta(u_i)=Y_i \int_{(b_{i-1},b_i] \times \bR_0} z  \widehat{N}(dx,dz)=Y_i L\big((b_{i-1},b_i] \big)=I(Y_i 1_{(b_{i-1},b_i]}).
\]

Note that $V_{\Phi}(x,z)=z\Phi(x)=\sum_{i=1}^{n}Y_iu_i(x,z)$. Hence, $V_{\Phi} \in {\rm dom}(\delta)$ and
\[
\delta(V_{\Phi})=\sum_{i=1}^{n} \delta(Y_i u_i)=\sum_{i=1}^{n}Y_i L\big((b_{i-1},b_i] \big)=I(\Phi).
\]

\medskip

{\em Step 2.} Assume that $\Phi \in \cE$ is of form \eqref{simple}. Fix $i\in \{1,\ldots, n\}$. Let $u_i$ be as in Step 1. Since ${\rm dom}(D)$ is dense in $L^2(\Omega)$, there exists a sequence $\{Y_i^{k}\}_{k\geq 1} \subset {\rm dom}(D)$ such that $\bE|Y_i^k-Y_i|^2 \to \infty$ as $k \to \infty$. 
By {\em Step 1}, for any $k\geq 1$, $Y_i^k u_i \in {\rm dom}(\delta)$ and 
\begin{equation}
\label{dI}
\delta(Y_i^k u_i) =I(Y_i^k 1_{(b_{i-1},b_i]}).
\end{equation} 
Note that  $Y_i^k u_i \to Y_i u_i$ in $L^2(\Omega;\cH)$ as $k \to \infty$, since
\[
\bE\|Y_i^k u_i - Y_i u_i\|_{\cH}^2=\bE\int_{\bR \times \bR_{0}} |Y_i^k-Y_i|^2 u_i^2(x,z)dx\nu(dz)=m_2 (b_i-b_{i-1})\bE|Y_i^k-Y_i|^2.
\]
Since $\delta$ is a closed operator, $\delta(Y_i^k u_i) \to \delta(Y_i u_i)$ in $L^2(\Omega)$, as $k \to \infty$.
Letting $k \to \infty$ in \eqref{dI}, we obtain that $\delta(Y_i u_i) =I(Y_i 1_{(b_{i-1},b_i]})$. The conclusion follows.

\medskip

{\em Step 3.} Let $\Phi \in \cL^*$ be arbitrary. By the construction of the integral, there exists a sequence $\{\Phi_n\}_{n\geq 1}$ in $\cE$ such that
\[
\bE\int_{\bR}|\Phi_n(x)-\Phi(x)|^2 dx \to 0 \quad \mbox{as} \quad n\to \infty.
\]
By {\em Step 2}, $\delta(V_{\Phi_n})=I(\Phi_n)$ for any $n\geq 1$. The conclusion follows letting $n \to \infty$ in $L^2(\Omega)$, using the fact that $\delta$ is a closed operator and $V_{\Phi_n} \to V_{\Phi}$ in $L^2(\Omega;\cH)$.

\end{proof}

{\em Acknowledgement.} The authors would like to thank Jinxin Wang for his help with the proof of Lemma 
\ref{lem-p-mom-L}.

\end{document}